\documentclass[12pt]{article}

\usepackage{amsmath, epsfig, cite}
\usepackage{amssymb}
\usepackage{amsfonts}
\usepackage{latexsym}
\usepackage{amsthm}

\newtheorem{thm}{Theorem}[section]

\newtheorem{cor}[thm]{Corollary}

\newtheorem{lem}[thm]{Lemma}

\numberwithin{equation}{section}

\renewcommand{\thefootnote}{}

\begin{document}

\begin{center}
{\large\bf Two parametric $q$-supercongruences from a summation
formula for $q$-series

 \footnote{ Corresponding author$^*$. Email addresses: weichuanan78@163.com (C. Wei), 378075837@qq.com (G. Ruan)}}
\end{center}

\renewcommand{\thefootnote}{$\dagger$}

\vskip 2mm \centerline{Chuanan Wei, Guozhu Ruan$^*$}
\begin{center}
{School of Intelligent Medicine and Technology\\ Hainan Medical
University, Haikou 571199, China}
\end{center}

%%date: January 4, 2011
%\vskip 5mm
%\noindent {\it Suggested Running title}: Two Identities of Gould

\vskip 0.7cm \noindent{\bf Abstract.}  With the help of a summation
formula for $q$-series and the creative microscoping method, we
shall establish two parametric $q$-supercongruences. They are both
modulo the third power of a cyclotomic polynomial. When $q\to1$, one
of them is able to engender the following conclusion: for any prime
$p\equiv2\pmod{3}$ and any nonnegative integer $s$ subject to $
s\leq (p-2)/3$,
\[\sum_{k=s}^{(p+1)/3+s}(6k-1)\frac{(-\frac{1}{3})_{k-s}(-\frac{1}{3})_{k+s}(-\frac{1}{3})_{k}}{(k-s)!(k+s)!k!}
\equiv 0\pmod{p^3}.\]

\vskip 3mm \noindent {\it Keywords}: supercongruence;
$q$-supercongruence; a summation formula for $q$-series; creative
microscoping method

 \vskip 0.2cm \noindent{\it AMS
Subject Classifications:} 33D15; 11A07; 11B65

\section{Introduction}
For a complex variable $x$ and a nonnegative integer $n$, define the
shifted-factorial by
\[(x)_{0}=1\quad \text{and}\quad (x)_{n}
=x(x+1)\cdots(x+n-1)\quad \text{when}\quad n\in\mathbb{Z}^{+}.\] Let
$p$ be a prime and let $\mathbb{Z}_p$ denote the ring of all
$p$-adic integers. Define Morita's $p$-adic Gamma function (cf.
\cite[Chapter 7]{Robert}) to be
 \[\Gamma_{p}(0)=1\quad \text{and}\quad \Gamma_{p}(n)
=(-1)^n\prod_{\substack{1\leqslant k< n\\
p\nmid k}}k,\quad \text{when}\quad n\in\mathbb{Z}^{+}.\] Observing
that $\mathbb{N}$ is a dense subset of $\mathbb{Z}_p$ related to the
$p$-adic norm $|\cdot|_p$, for each $x\in\mathbb{Z}_p$, the
definition of $p$-adic Gamma function can be extended as
 \[\Gamma_{p}(x)
=\lim_{\substack{n\in\mathbb{N}\\
|x-n|_p\to0}}\Gamma_{p}(n).\]

 1997, Van Hamme \cite[(F.2)]{Hamme} conjectured the interesting
supercongruence: for any prime $p\equiv1\pmod{6}$,
\begin{align}
\sum_{k=0}^{(p-1)/3}(-1)^k(6k+1)\frac{(\frac{1}{3})_k^3 }{k!^3}
\equiv p\pmod{p^{3}}.
 \label{Hamme-a}
\end{align}
After some years, Swisher  \cite{Swisher} employed Long's method to
prove  \eqref{Hamme-a} and several other results.
 In 2017, He \cite{He} discovered the following supercongruence associated with \eqref{Hamme-a}: for any prime $p$,
\begin{equation}\label{eq:he}
\sum_{k=0}^{p-1}(6k+1)\frac{(\frac{1}{3})_k^3}{k!^3}\equiv
\begin{cases} \displaystyle p\Gamma_p\big(\tfrac{2}{3}\big)^3  \pmod{p^2}, &\text{if $p\equiv 1\pmod 6$,}\\[10pt]
\displaystyle -6\Gamma_p\big(\tfrac{2}{3}\big)^3\pmod{p^2},
&\text{if $p\equiv 5\pmod 6$.}
\end{cases}
\end{equation}
Some years later, Guo and Schlosser \cite{GS23} proved that
\eqref{eq:he} is true modulo $p^3$ in the $p\equiv 1\pmod 6$ case
and Wang and Sun \cite{Wang} showed that \eqref{eq:he} is right
modulo $p^3$ in the $p\equiv 5\pmod 6$ case. Namely,
\begin{equation}\label{eq:Gw}
\sum_{k=0}^{p-1}(6k+1)\frac{(\frac{1}{3})_k^3}{k!^3}\equiv
\begin{cases} \displaystyle p\Gamma_p\big(\tfrac{2}{3}\big)^3  \pmod{p^3}, &\text{if $p\equiv 1\pmod 6$,}\\[10pt]
\displaystyle -6\Gamma_p\big(\tfrac{2}{3}\big)^3\pmod{p^3},
&\text{if $p\equiv 5\pmod 6$.}
\end{cases}
\end{equation}

 For two complex numbers $x$ and $q$ with $|q|<1$ and a nonnegative integer $n$, define the $q$-shifted factorial
 as
 \begin{equation*}
(x;q)_{\infty}=\prod_{k=0}^{\infty}(1-xq^k)\quad\text{and}\quad
(x;q)_n=\frac{(x;q)_{\infty}}{(xq^n;q)_{\infty}}\quad
\text{when}\quad n\in\mathbb{Z}^{+}\cup\{0\}.
 \end{equation*}
For the aim of simplicity, we usually adopt the compact notation:
\begin{equation*}
(x_1,x_2,\dots,x_m;q)_{n}=(x_1;q)_{n}(x_2;q)_{n}\cdots(x_m;q)_{n},
 \end{equation*}
where $m\in\mathbb{Z}^{+}$ and $n\in\mathbb{Z}^{+}\cup\{0,\infty\}.$

In 2019, Guo \cite{Guo-19} established the following $q$-analogue of
\eqref{Hamme-a}: for any positive integer $n\equiv 1\pmod 3$,
\begin{align*}
\sum_{k=0}^{(n-1)/3}(-1)^k[6k+1]\frac{(q;q^3)_k^3}{(q^3;q^3)_k^3}q^{(3k^2+k)/2}
\equiv (-1)^{(n-1)/3}[n]q^{(n-1)(n-2)/6}\pmod{[n]\Phi_n(q)^2}.
\end{align*}
Here and all over the paper,  $[r]=[r]_q$ stands for the $q$-integer
$(1-q^r)/(1-q)$ and $\Phi_n(q)$ represents the $n$-th cyclotomic
polynomial in $q$:
\begin{equation*}
\Phi_n(q)=\prod_{\substack{1\leqslant k\leqslant n\\
\gcd(k,n)=1}}(q-\zeta^k),
\end{equation*}
where $\zeta$ is an $n$-th primitive root of unity.
 In 2023, Guo and
Schlosser \cite{GS23} gave another $q$-analogue of \eqref{Hamme-a}:
for any positive integer $n\equiv 1\pmod 6$,
\begin{align*}
\sum_{k=0}^{(n-1)/3}(-1)^k[6k+1]\frac{(q;q^3)_k^3}{(q^3;q^3)_k^3}
\equiv
[n]q^{2(1-n)/3}\frac{(-q^3;q^3)_{(n-1)/3}}{(-q^2;q^3)_{(n-1)/3}}\pmod{[n]\Phi_n(q)^2},
\end{align*}
and displayed the following $q$-supercongruence: for any positive
integer $n\equiv 1\pmod 3$,
\begin{align*}
\sum_{k=0}^{n-1}[6k+1]\frac{(q;q^3)_k^3}{(q^3;q^3)_k^3} \equiv
[n]q^{2(1-n)/3}\frac{(q^3;q^3)_{(n-1)/3}}{(q^2;q^3)_{(n-1)/3}}\pmod{\Phi_n(q)^3},
\end{align*}
which is a $q$-analogue of \eqref{eq:Gw} in the $p\equiv 1\pmod 6$
case. In the same paper, they also provided the following
$q$-analogue of \eqref{eq:he} in the $p\equiv 5\pmod 6$ case: for
any positive integer $n\equiv 2\pmod 3$,
\begin{align}\label{eq:guo-a}
\sum_{k=0}^{n-1}[6k+1]\frac{(q;q^3)_k^3}{(q^3;q^3)_k^3} \equiv
[2n]q^{2(1-2n)/3}\frac{(q^3;q^3)_{(2n-1)/3}}{(q^2;q^3)_{(2n-1)/3}}\pmod{\Phi_n(q)^2}.
\end{align}
Recently, Guo and Zhu \cite{GuoZhu-a} displayed the following
generalization of \eqref{eq:guo-a}: for any positive integer
$n\equiv 2\pmod 3$,
\begin{align*}
\sum_{k=0}^{n-1}[6k+1]\frac{(q;q^3)_k^3}{(q^3;q^3)_k^3}
&\equiv
[2n]q^{2(1-2n)/3}\frac{(q^3;q^3)_{(2n-1)/3}}{(q^2;q^3)_{(2n-1)/3}}
\\[5pt]
&\quad\times\bigg\{1-4[n]^2\sum_{i=1}^{(n-2)/3}\frac{q^{3i}}{[3i]^2}
\bigg\}\pmod{\Phi_n(q)^3},
\end{align*}
which is a $q$-analogue of \eqref{eq:Gw} in the $p\equiv 5\pmod 6$
case. For more $q$-analogues of supercongruences, we refer the
reader to the papers
\cite{Guo-a2,Guo2023,GuoZhu,GuoZhu-b,Li,LW,Wei-c}.

Motivated by the works just mentioned, we shall establish the
following two  parametric $q$-supercongruences.

\begin{thm}\label{thm-a}
Let $n,s$ be nonnegative integers such that $n\equiv 2\pmod 3$ and
$s\leq(n-2)/3$. Then
\begin{align}\label{eq:wei-a}
&\sum_{k=s}^{(n+1)/3+s}[6k-1]\frac{(q^{-1};q^3)_{k-s}(q^{-1};q^3)_{k+s}(q^{-1};q^3)_{k}}
{(q^{3};q^3)_{k-s}(q^{3};q^3)_{k+s}(q^{3};q^3)_{k}}q^{3k}
\equiv0\pmod {\Phi_n(q)^3}.
\end{align}
\end{thm}

Setting $n=p$ and taking $q\to 1$ in Theorem \ref{thm-a}, we catch
hold of the following conclusion.

\begin{cor}\label{cor-a}
Let $p$ be a prime and let $s$ be a nonnegative integer subject to
$p\equiv2\pmod{3}$ and $ s\leq (p-2)/3$. Then
\[\sum_{k=s}^{(p+1)/3+s}(6k-1)\frac{(-\frac{1}{3})_{k-s}(-\frac{1}{3})_{k+s}(-\frac{1}{3})_{k}}{(k-s)!(k+s)!k!}
\equiv 0\pmod{p^3}.\]
\end{cor}

\begin{thm}\label{thm-b}
Let $n>1,s$ be nonnegative integers such that $n\equiv 1\pmod 3$ and
$s\leq(n-4)/6$. Then
\begin{align}\label{eq:wei-b}
&\sum_{k=s}^{(2n+1)/3+s}[6k-1]\frac{(q^{-1};q^3)_{k-s}(q^{-1};q^3)_{k+s}(q^{-1};q^3)_{k}}
{(q^{3};q^3)_{k-s}(q^{3};q^3)_{k+s}(q^{3};q^3)_{k}}q^{3k}
\equiv0\pmod {\Phi_n(q)^3}.
\end{align}
\end{thm}

Fixing $n=p$ and taking $q\to 1$ in Theorem \ref{thm-b}, we grasp
hold of the following conclusion.

\begin{cor}\label{cor-b}
Let $p$ be a prime and let $s$ be a nonnegative integer subject to
$p\equiv1\pmod{3}$ and $ s\leq (p-4)/6$. Then
\[\sum_{k=s}^{(2p+1)/3+s}(6k-1)\frac{(-\frac{1}{3})_{k-s}(-\frac{1}{3})_{k+s}(-\frac{1}{3})_{k}}{(k-s)!(k+s)!k!}
\equiv 0\pmod{p^3}.\]
\end{cor}

 Following Gasper and Rahman \cite{Gasper}, define the $q$-series (basis hypergeometric series)
 to be
$$
_{r+1}\phi_{r}\left[\begin{array}{c}
a_1,a_2,\ldots,a_{r+1}\\
b_1,b_2,\ldots,b_{r}
\end{array};q,\, z
\right] =\sum_{k=0}^{\infty}\frac{(a_1,a_2,\ldots, a_{r+1};q)_k}
{(q,b_1,b_2,\ldots,b_{r};q)_k}z^k.
$$
Then a summation formula for $q$-series (cf. \cite[Appendix
(II.20)]{Gasper}) can be stated as
\begin{align}
_{6}\phi_{5}\!\left[\begin{array}{c}
a,\,qa^{\frac{1}{2}},\, -qa^{\frac{1}{2}},\, b,\, c,\, d \\
a^{\frac{1}{2}},\,-a^{\frac{1}{2}},\, aq/b,\, aq/c,\, aq/d
\end{array};q,\,\frac{aq}{bcd}  \right]
=\frac{(aq,aq/bc,aq/bd,aq/cd;q)_{\infty}}{(aq/b,aq/c,aq/d,aq/bcd;q)_{\infty}},
\label{Dixon}
\end{align}
where $|aq/bcd|<1$.

The rest of the paper is arranged as follows. In terms of the
summation formula for $q$-series \eqref{Dixon} and the creative
microscoping method introduced in \cite{GuoZu}, we are going to
prove Theorem \ref{thm-a} in Section 2. Similarly, the proof of
Theorem \ref{thm-b} will be displayed in Section 3.

%%%%%%%%%%%%%%%%%%%%%%%%%%%%%%%%%%%%%%%%%%%%%%%%%%%%%%%%%%%%%%%%%%%%%%%%%%%%%%%%%%%%%%%%%%%%%%%%%%%%%%%%%%%%%%%%%%%%%%%%%%%%%%%%%%%%%%%%%%%%%%%%%%%%%%%%%%%%%%%%%%%%%%%%%%%%%%%%%%%%%%%%%%%%%%%%%%%%%%%%%%%%%%%%%%%
\section{Proof of Theorem \ref{thm-a}}
%%%%%%%%%%%%%%%%%%%%%%%%%%%%%%%%%%%%%%%%%%%%%%%%%%%%%%%%%%%%%%%%%%%%%%%%%%%%%%%%%%%%%%%%%%%%%%%%%%%%%%%%%%%%%%%%%

In order to prove Theorem \ref{thm-a}, we need the following three
lemmas.

\begin{lem}\label{lem-a}
Let $n,s$ be nonnegative integers such that $n\equiv 2\pmod 3$. Then
\begin{align}\label{eq:wei-aa}
&\sum_{k=s}^{(n+1)/3+s}[6k-1]\frac{(aq^{-1};q^3)_{k-s}(q^{-1};q^3)_{k+s}(q^{-1}/a;q^3)_{k}}
{(q^{3};q^3)_{k-s}(q^{3}/a;q^3)_{k+s}(aq^{3};q^3)_{k}}q^{3k}
\equiv0\pmod {\Phi_n(q)}.
\end{align}
\end{lem}

\begin{proof}
Firstly, we shall prove that \eqref{eq:wei-aa} is true for
$s\leq(n+1)/6$. Recall the $q$-congruence from Guo and Schlosser
\cite[Lemma 4]{GS3} : for $0\leqslant k\leqslant m$,
\begin{equation*}
\frac{(aq^{-1};q^d)_{m-k}}{(q^d/a;q^d)_{m-k}} \equiv
(-a)^{m-2k}\frac{(aq^{-1};q^d)_k}{(q^d/a;q^d)_k}
q^{m(dm-d-2)/2+(d+1)k} \pmod{\Phi_n(q)},
\end{equation*}
where $d,m,n$ are positive integers with $m\leq n-1$ and
$dm\equiv1\pmod{n}$.
 The $d=3,m=(n+1)/3$ case of it reads
\begin{align}
&\frac{(aq^{-1};q^3)_{(n+1)/3-k}}{(q^3/a;q^3)_{(n+1)/3-k}} \equiv
(-a)^{(n+1)/3-2k}\frac{(aq^{-1};q^3)_k}{(q^3/a;q^3)_k}
q^{(n+1)(n-4)/6+4k} \pmod{\Phi_n(q)}.
 \label{eq:wei-bb}
\end{align}
Then we can proceed as follows: for $s\leqslant k\leqslant
(n+1)/3-s$,
\begin{align}
&\frac{(aq^{-1};q^3)_{(n+1)/3-k+s}}{(q^{3}/a;q^3)_{(n+1)/3-k-s}}
  \notag\\[1mm]
&\quad=\frac{(aq^{-1};q^3)_{(n+1)/3-k-s}}{(q^{3}/a;q^3)_{(n+1)/3-k-s}}
\notag\\[1mm]
&\qquad\times(1-aq^{n-3k-3s})(1-aq^{n-3k-3s+3})\cdots
(1-aq^{n-3k+3s-3})
 \notag\\[1mm]
&\quad\equiv
(-a)^{(n+1)/3-2k-2s}\frac{(aq^{-1};q^3)_{k+s}}{(q^3/a;q^3)_{k+s}}
q^{(n+1)(n-4)/6+4k+4s}
\notag\\[1mm]
&\qquad\times(1-aq^{-3k-3s})(1-aq^{-3k-3s+3})\cdots
(1-aq^{-3k+3s-3})
 \notag\\[1mm]
&\quad=(-a)^{(n+1)/3-2k}\frac{(aq^{-1};q^3)_{k+s}}{(q^3/a;q^3)_{k-s}}q^{(n+1)(n-4)/6+s+4k-6ks}
\pmod{\Phi_n(q)},   \label{eq:wei-cc}
\end{align}
where we have used $q^n\equiv1\pmod{\Phi_n(q)}$. Similarly, we find
also the following relation: for $s\leqslant k\leqslant (n+1)/3-s$,
\begin{align}
&\frac{(aq^{-1};q^3)_{(n+1)/3-k-s}}{(q^{3}/a;q^3)_{(n+1)/4-k+s}}
  \notag\\
&\quad\equiv
(-a)^{(n+1)/3-2k}\frac{(aq^{-1};q^3)_{k-s}}{(q^3/a;q^3)_{k+s}}q^{(n+1)(n-4)/6-s+4k+6ks}
\pmod{\Phi_n(q)}.    \label{eq:wei-dd}
\end{align}
Through \eqref{eq:wei-bb}-\eqref{eq:wei-dd}, there holds the coming
$q$-congruence: for $M=(n+1)/3$ and $s\leqslant k\leqslant M-s$,
\begin{align*}
&[6(M-k)-1]\frac{(aq^{-1};q^3)_{M-k-s}(q^{-1};q^3)_{M-k+s}(q^{-1}/a;q^3)_{M-k}}
{(q^3;q^3)_{M-k-s}(q^3/a;q^4)_{M-k+s}(aq^3;q^3)_{M-k}} q^{3(M-k)}
\\&\quad
\equiv
-[6k-1]\frac{(aq^{-1};q^3)_{k-s}(q^{-1};q^3)_{k+s}(q^{-1}/a;q^3)_{k}}
{(q^{3};q^3)_{k-s}(q^{3}/a;q^3)_{k+s}(aq^{3};q^3)_{k}}q^{3k}
\pmod{\Phi_n(q)},
\end{align*}
where we have utilized $q^{n/2}\equiv-1\pmod{\Phi_n(q)}$ for any
positive even integer $n$. The last formula indicates that
\begin{align*}
&\sum_{k=s}^{(n+1)/3-s}[6k-1]\frac{(aq^{-1};q^3)_{k-s}(q^{-1};q^3)_{k+s}(q^{-1}/a;q^3)_{k}}
{(q^{3};q^3)_{k-s}(q^{3}/a;q^3)_{k+s}(aq^{3};q^3)_{k}}q^{3k}
\equiv0\pmod {\Phi_n(q)}.
\end{align*}
When $(n+1)/3-s<k\leq (n+1)/3+s$, it is easy to see that the factor
$(1-q^{n})$ appears in the $q$-shifted factorial
$(q^{-1};q^3)_{k+s}$ and so \eqref{eq:wei-aa} is true for
$s\leq(n+1)/6$.

Secondly, we shall prove that \eqref{eq:wei-aa} is true for
$s>(n+1)/6$. To achieve this goal, it is sufficient to prove the
following equation:
\begin{align}\label{eq:wei-ee}
&\sum_{k=s}^{(n+1)/3+s}[6k-1]\frac{(aq^{-1};q^3)_{k-s}(q^{-1-n};q^3)_{k+s}(q^{-1}/a;q^3)_{k}}
{(q^{3};q^3)_{k-s}(q^{3}/a;q^3)_{k+s}(aq^{3};q^3)_{k}}q^{3k} =0.
\end{align}
Via some evaluation, we discover that the expression on the
left-hand side of \eqref{eq:wei-ee} equals
\begin{align*}
&\sum_{k=s}^{(n+1)/3+s}[6k-1]\frac{(aq^{-1};q^3)_{k-s}(q^{-1-n};q^3)_{k+s}(q^{-1}/a;q^3)_{k}}
{(q^{3};q^3)_{k-s}(q^{3}/a;q^3)_{k+s}(aq^{3};q^3)_{k}}q^{3k}
\\&\quad=
\sum_{k=0}^{(n+1)/3}[6k+6s-1]\frac{(aq^{-1};q^3)_{k}(q^{-1-n};q^3)_{k+2s}(q^{-1}/a;q^3)_{k+s}}
{(q^{3};q^3)_{k}(q^{3}/a;q^3)_{k+2s}(aq^{3};q^3)_{k+s}}q^{3k+3s}
\\&\quad=
\sum_{k=0}^{(n+1)/3-2s}[6k+6s-1]\frac{(aq^{-1};q^3)_{k}(q^{-1-n};q^3)_{k+2s}(q^{-1}/a;q^3)_{k+s}}
{(q^{3};q^3)_{k}(q^{3}/a;q^3)_{k+2s}(aq^{3};q^3)_{k+s}}q^{3k+3s}
\\&\quad=0,
\end{align*}
where we have employed the fact $\sum_{k=0}^{-m}=0$ for any positive
integer $m$. Thus \eqref{eq:wei-aa} is correct for $s>(n+1)/6$.
\end{proof}

\begin{lem}\label{lem-b}
Let $n,s$ be nonnegative integers such that $n\equiv 2\pmod 3$ and
$s\leq (n-2)/3$. Then modulo $a-q^n$,
\begin{align}\label{eq:wei-ff}
&\sum_{k=s}^{(n+1)/3+s}[6k-1]\frac{(aq^{-1};q^3)_{k-s}(q^{-1};q^3)_{k+s}(q^{-1}/a;q^3)_{k}}
{(q^{3};q^3)_{k-s}(q^{3}/a;q^3)_{k+s}(aq^{3};q^3)_{k}}q^{3k}
\equiv0.
\end{align}
\end{lem}

\begin{proof}
The $a=q^n$ case of the expression on the left-hand side of
\eqref{eq:wei-ff} is equal to
\begin{align}\label{eq:wei-gg}
&\sum_{k=s}^{(n+1)/3+s}[6k-1]\frac{(q^{-1+n};q^3)_{k-s}(q^{-1};q^3)_{k+s}(q^{-1-n};q^3)_{k}}
{(q^{3};q^3)_{k-s}(q^{3-n};q^3)_{k+s}(q^{3+n};q^3)_{k}}q^{3k}
\notag\\&\quad=
\sum_{k=0}^{(n+1)/3}[6k+6s-1]\frac{(q^{-1+n};q^3)_{k}(q^{-1};q^3)_{k+2s}(q^{-1-n};q^3)_{k+s}}
{(q^{3};q^3)_{k}(q^{3-n};q^3)_{k+2s}(q^{3+n};q^3)_{k+s}}q^{3k+3s}
\notag\\&\quad=
 [6s-1]\frac{(q^{-1};q^3)_{2s}(q^{-1-n};q^3)_{s}}
{(q^{3-n};q^3)_{2s}(q^{3+n};q^3)_{s}}q^{3s}
 \notag\\&\qquad\times
\sum_{k=0}^{(n+1)/3-s}\frac{1-q^{6s-1+6k}}{1-q^{6s-1}}
\frac{(q^{6s-1},q^{-1+n},q^{-1-n+3s};q^3)_{k}}
{(q^{3},q^{3-n+6s},q^{3+n+3s};q^3)_{k}}q^{3k}.
\end{align}
Conducting the replacements $a\mapsto q^{6s-1}$, $b\mapsto
q^{-1+n}$, $c\mapsto q^{-1-n+3s}$, $d\mapsto q^{3s+1}$, $q\mapsto
q^{3}$ in the identity \eqref{Dixon}, there is
\begin{align}
&\sum_{k=0}^{(n+1)/3-s}\frac{1-q^{6s-1+6k}}{1-q^{6s-1}}
\frac{(q^{6s-1},q^{-1+n},q^{-1-n+3s};q^3)_{k}}
{(q^{3},q^{3-n+6s},a^{3+n+3s};q^3)_{k}}q^{3k}
 \notag \\&\quad\:=
\frac{(q^{2-n+3s},q^{6s+2};q^3)_{(n+1)/3-s}}{(q^{3-n+6s},q^{3s+1};q^3)_{(n+1)/3-s}}
 \notag\\&\quad\:=0.
 \label{eq:wei-hh}
\end{align}
Substituting \eqref{eq:wei-hh} into \eqref{eq:wei-gg}, we obtain the
following equation:
\begin{align*}
&\sum_{k=s}^{(n+1)/3+s}[6k-1]\frac{(q^{-1+n};q^3)_{k-s}(q^{-1};q^3)_{k+s}(q^{-1-n};q^3)_{k}}
{(q^{3};q^3)_{k-s}(q^{3-n};q^3)_{k+s}(q^{3+n};q^3)_{k}}q^{3k} =0.
\end{align*}
Hence we arrive at the $q$-congruence \eqref{eq:wei-ff} to complete
the proof.
\end{proof}

\begin{lem}\label{lem-c}
Let $n,s$ be nonnegative integers such that $n\equiv 2\pmod 3$. Then
modulo $1-aq^n$,
\begin{align}\label{eq:wei-ii}
&\sum_{k=s}^{(n+1)/3+s}[6k-1]\frac{(aq^{-1};q^3)_{k-s}(q^{-1};q^3)_{k+s}(q^{-1}/a;q^3)_{k}}
{(q^{3};q^3)_{k-s}(q^{3}/a;q^3)_{k+s}(aq^{3};q^3)_{k}}q^{3k}
\equiv0.
\end{align}
\end{lem}

\begin{proof}
The $a=q^{-n}$ case of the the expression on the left-hand side of
\eqref{eq:wei-ii} equals
\begin{align}\label{eq:wei-jj}
&\sum_{k=s}^{(n+1)/3+s}[6k-1]\frac{(q^{-1-n};q^3)_{k-s}(q^{-1};q^3)_{k+s}(q^{-1+n};q^3)_{k}}
{(q^{3};q^3)_{k-s}(q^{3+n};q^3)_{k+s}(q^{3-n};q^3)_{k}}q^{3k}
\notag\\&\quad=
\sum_{k=0}^{(n+1)/3}[6k+6s-1]\frac{(q^{-1-n};q^3)_{k}(q^{-1};q^3)_{k+2s}(q^{-1+n};q^3)_{k+s}}
{(q^{3};q^3)_{k}(q^{3+n};q^3)_{k+2s}(q^{3-n};q^3)_{k+s}}q^{3k+3s}
\notag\\&\quad=
 [6s-1]\frac{(q^{-1};q^3)_{2s}(q^{-1+n};q^3)_{s}}
{(q^{3+n};q^3)_{2s}(q^{3-n};q^3)_{s}}q^{3s}
 \notag\\&\qquad\times
\sum_{k=0}^{(n+1)/3}\frac{1-q^{6s-1+6k}}{1-q^{6s-1}}
\frac{(q^{6s-1},q^{-1-n},q^{-1+n+3s};q^3)_{k}}
{(q^{3},q^{3+n+6s},q^{3-n+3s};q^3)_{k}}q^{3k}.
\end{align}
Making the replacements $a\mapsto q^{6s-1}$, $b\mapsto q^{-1-n}$,
$c\mapsto q^{-1+n+3s}$, $d\mapsto q^{3s+1}$, $q\mapsto q^{3}$ in the
identity \eqref{Dixon}, we have
\begin{align}
&\sum_{k=0}^{(n+1)/3}\frac{1-q^{6s-1+6k}}{1-q^{6s-1}}
\frac{(q^{6s-1},q^{-1-n},q^{-1+n+3s};q^3)_{k}}
{(q^{3},q^{3+n+6s},q^{3-n+3s};q^3)_{k}}q^{3k}
 \notag \\&\quad\:=
\frac{(q^{2-n},q^{6s+2};q^3)_{(n+1)/3}}{(q^{3-n+3s},q^{3s+1};q^3)_{(n+1)/3}}
 \notag\\&\quad\:=0.
 \label{eq:wei-kk}
\end{align}
Substituting \eqref{eq:wei-kk} into \eqref{eq:wei-jj}, we get the
following equation:
\begin{align*}
&\sum_{k=s}^{(n+1)/3+s}[6k-1]\frac{(q^{-1-n};q^3)_{k-s}(q^{-1};q^3)_{k+s}(q^{-1+n};q^3)_{k}}
{(q^{3};q^3)_{k-s}(q^{3+n};q^3)_{k+s}(q^{3-n};q^3)_{k}}q^{3k} =0.
\end{align*}
Therefore, we are led to the $q$-congruence \eqref{eq:wei-ii} to
complete the proof.
\end{proof}

Now we are ready to prove Theorem \ref{thm-a}.

\begin{proof}[Proof of Theorem \ref{thm-a}]
Considering that the polynomials $\Phi_n(q)$, $a-q^{n}$, and
$1-aq^{n}$ in $q$ are pairwise relatively prime, we can derive, from
\eqref{eq:wei-aa}, \eqref{eq:wei-ff}, and \eqref{eq:wei-ii}, the
following  $q$-supercongruence: for $s\leq (n-2)/3$, modulo
$\Phi_n(q)(a-q^n)(1-aq^n)$,
\begin{align*}
&\sum_{k=s}^{(n+1)/3+s}[6k-1]\frac{(aq^{-1};q^3)_{k-s}(q^{-1};q^3)_{k+s}(q^{-1}/a;q^3)_{k}}
{(q^{3};q^3)_{k-s}(q^{3}/a;q^3)_{k+s}(aq^{3};q^3)_{k}}q^{3k}
\equiv0.
\end{align*}
Taking $a=1$ in this $q$-supercongruence, we can deduce
\eqref{eq:wei-a} in the end.
\end{proof}

%%%%%%%%%%%%%%%%%%%%%%%%%%%%%%%%%%%%%%%%%%%%%%%%%%%%%%%%%%%%%%%%%%%%%%%%%%%%%%%%%%%%%%%%%%%%%%%%%%%%%%%%%%%%%%%%%%%%%%%%%%%%%%%%%%%%%%%%%%%%%%%%%%%%%%%%%%%%%%%%%%%%%%%%%%%%%%%%%%%%%%%%%%%%%%%%%%%%%%%%%%%%%%%%%%%
\section{Proof of Theorem \ref{thm-b}}
%%%%%%%%%%%%%%%%%%%%%%%%%%%%%%%%%%%%%%%%%%%%%%%%%%%%%%%%%%%%%%%%%%%%%%%%%%%%%%%%%%%%%%%%%%%%%%%%%%%%%%%%%%%%%%%%%

For the aim to prove Theorem \ref{thm-b}, we require the following
three lemmas.

\begin{lem}\label{lem-d}
Let $n>1,s$ be nonnegative integers such that $n\equiv 1\pmod 3$ and
$s\leq (n-4)/3$. Then
\begin{align}\label{eq:wei-aaa}
&\sum_{k=s}^{(2n+1)/3+s}[6k-1]\frac{(aq^{-1};q^3)_{k-s}(q^{-1};q^3)_{k+s}(q^{-1}/a;q^3)_{k}}
{(q^{3};q^3)_{k-s}(q^{3}/a;q^3)_{k+s}(aq^{3};q^3)_{k}}q^{3k}
\equiv0\pmod {\Phi_n(q)}.
\end{align}
\end{lem}

\begin{proof}
Conducting the replacements $a\mapsto q^{-1-2n+6s}$, $b\mapsto
aq^{-1}$, $c\mapsto q^{3s-1}/a$, $d\mapsto q^{1-n+3s}$, $q\mapsto
q^{3}$ in the identity \eqref{Dixon}, it is not difficult to
understand that
\begin{align*}
&\sum_{k=0}^{(2n+1)/3-2s}\frac{1-q^{-1-2n+6s+6k}}{1-q^{-1-2n+6s}}
\frac{(q^{-1-2n+6s},aq^{-1},q^{3s-1}/a;q^3)_{k}}
{(q^{3},q^{3-2n+6s}/a,a^{3-2n+3s};q^3)_{k}}q^{(3-n)k}
 \notag \\&\quad\:=
\frac{2(q^{4-2n+3s},q^{2-2n+6s};q^3)_{(n-1)/3-s}}{(aq^{3-2n+3s},q^{3-2n+6s}/a;q^3)_{(n-1)/3-s}}
 \notag\\[1mm]&\quad\:=0.
\end{align*}
Then we can calculate as follows:
\begin{align*}
&\sum_{k=s}^{(2n+1)/3+s}[6k-1-2n]\frac{(aq^{-1};q^3)_{k-s}(q^{-1-2n};q^3)_{k+s}(q^{-1}/a;q^3)_{k}}
{(q^{3};q^3)_{k-s}(q^{3-2n}/a;q^3)_{k+s}(aq^{3-2n};q^3)_{k}}q^{(3-n)k}
\notag\\&\quad=
\sum_{k=0}^{(2n+1)/3}[6k+6s-1-2n]\frac{(aq^{-1};q^3)_{k}(q^{-1-2n};q^3)_{k+2s}(q^{-1}/a;q^3)_{k+s}}
{(q^{3};q^3)_{k}(q^{3-2n}/a;q^3)_{k+2s}(aq^{3-2n};q^3)_{k+s}}q^{(3-n)(k+s)}
\notag
\end{align*}
\begin{align*}
\notag\\&\quad=
 [6s-1-2n]\frac{(q^{-1-2n};q^3)_{2s}(q^{-1}/a;q^3)_{s}}
{(q^{3-2n}/a;q^3)_{2s}(aq^{3-2n};q^3)_{s}}q^{(3-n)s}
 \notag\\&\qquad\times
\sum_{k=0}^{(2n+1)/3-2s}\frac{1-q^{-1-2n+6s+6k}}{1-q^{-1-2n+6s}}
\frac{(q^{-1-2n+6s},aq^{-1},q^{3s-1}/a;q^3)_{k}}
{(q^{3},q^{3-2n+6s}/a,a^{3-2n+3s};q^3)_{k}}q^{(3-n)k}
\notag\\[1mm]&\quad=0.
\end{align*}
So it is routine to show the $q$-congruence \eqref{eq:wei-aaa} via
$q^{n}\equiv1\pmod{\Phi_n(q)}$.
\end{proof}

\begin{lem}\label{lem-b}
Let $n>1,s$ be nonnegative integers such that $n\equiv 1\pmod 3$ and
$s\leq (2n-2)/3$. Then modulo $a-q^{2n}$,
\begin{align}\label{eq:wei-bbb}
&\sum_{k=s}^{(2n+1)/3+s}[6k-1]\frac{(aq^{-1};q^3)_{k-s}(q^{-1};q^3)_{k+s}(q^{-1}/a;q^3)_{k}}
{(q^{3};q^3)_{k-s}(q^{3}/a;q^3)_{k+s}(aq^{3};q^3)_{k}}q^{3k}
\equiv0.
\end{align}
\end{lem}

\begin{proof}
The $a=q^{2n}$ case of the left-hand side of \eqref{eq:wei-bbb} is
equal to
\begin{align}\label{eq:wei-ccc}
&\sum_{k=s}^{(2n+1)/3+s}[6k-1]\frac{(q^{-1+2n};q^3)_{k-s}(q^{-1};q^3)_{k+s}(q^{-1-2n};q^3)_{k}}
{(q^{3};q^3)_{k-s}(q^{3-2n};q^3)_{k+s}(q^{3+2n};q^3)_{k}}q^{3k}
\notag\\&\quad=
\sum_{k=0}^{(2n+1)/3}[6k+6s-1]\frac{(q^{-1+2n};q^3)_{k}(q^{-1};q^3)_{k+2s}(q^{-1-2n};q^3)_{k+s}}
{(q^{3};q^3)_{k}(q^{3-2n};q^3)_{k+2s}(q^{3+2n};q^3)_{k+s}}q^{3k+3s}
\notag\\&\quad=
 [6s-1]\frac{(q^{-1};q^3)_{2s}(q^{-1-2n};q^3)_{s}}
{(q^{3-2n};q^3)_{2s}(q^{3+2n};q^3)_{s}}q^{3s}
 \notag\\&\qquad\times
\sum_{k=0}^{(2n+1)/3-s}\frac{1-q^{6s-1+6k}}{1-q^{6s-1}}
\frac{(q^{6s-1},q^{-1+2n},q^{-1-2n+3s};q^3)_{k}}
{(q^{3},q^{3-2n+6s},q^{3+2n+3s};q^3)_{k}}q^{3k}.
\end{align}
Performing the replacements $a\mapsto q^{6s-1}$, $b\mapsto
q^{-1+2n}$, $c\mapsto q^{-1-2n+3s}$, $d\mapsto q^{3s+1}$, $q\mapsto
q^{3}$ in the identity \eqref{Dixon}, we obtain
\begin{align}
&\sum_{k=0}^{(2n+1)/3-s}\frac{1-q^{6s-1+6k}}{1-q^{6s-1}}
\frac{(q^{6s-1},q^{-1+2n},q^{-1-2n+3s};q^3)_{k}}
{(q^{3},q^{3-2n+6s},a^{3+2n+3s};q^3)_{k}}q^{3k}
 \notag \\&\quad\:=
\frac{(q^{2-2n+3s},q^{6s+2};q^3)_{(2n+1)/3-s}}{(q^{3-2n+6s},q^{3s+1};q^3)_{(2n+1)/3-s}}
 \notag\\&\quad\:=0.
 \label{eq:wei-ddd}
\end{align}
Substituting \eqref{eq:wei-ddd} into \eqref{eq:wei-ccc}, there holds
the following equation:
\begin{align*}
&\sum_{k=s}^{(2n+1)/3+s}[6k-1]\frac{(q^{-1+2n};q^3)_{k-s}(q^{-1};q^3)_{k+s}(q^{-1-2n};q^3)_{k}}
{(q^{3};q^3)_{k-s}(q^{3-2n};q^3)_{k+s}(q^{3+2n};q^3)_{k}}q^{3k} =0.
\end{align*}
Thus we catch hold of the $q$-congruence \eqref{eq:wei-bbb} to
complete the proof.
\end{proof}

\begin{lem}\label{lem-f}
Let $n>1,s$ be nonnegative integers such that $n\equiv 1\pmod 3$.
Then modulo $1-aq^{2n}$,
\begin{align}\label{eq:wei-eee}
&\sum_{k=s}^{(2n+1)/3+s}[6k-1]\frac{(aq^{-1};q^3)_{k-s}(q^{-1};q^3)_{k+s}(q^{-1}/a;q^3)_{k}}
{(q^{3};q^3)_{k-s}(q^{3}/a;q^3)_{k+s}(aq^{3};q^3)_{k}}q^{3k}
\equiv0.
\end{align}
\end{lem}

\begin{proof}
The $a=q^{-2n}$ case of the left-hand side of \eqref{eq:wei-eee}
equals
\begin{align}\label{eq:wei-fff}
&\sum_{k=s}^{(2n+1)/3+s}[6k-1]\frac{(q^{-1-2n};q^3)_{k-s}(q^{-1};q^3)_{k+s}(q^{-1+2n};q^3)_{k}}
{(q^{3};q^3)_{k-s}(q^{3+2n};q^3)_{k+s}(q^{3-2n};q^3)_{k}}q^{3k}
\notag\\&\quad=
\sum_{k=0}^{(2n+1)/3}[6k+6s-1]\frac{(q^{-1-2n};q^3)_{k}(q^{-1};q^3)_{k+2s}(q^{-1+2n};q^3)_{k+s}}
{(q^{3};q^3)_{k}(q^{3+2n};q^3)_{k+2s}(q^{3-2n};q^3)_{k+s}}q^{3k+3s}
\notag\\&\quad=
 [6s-1]\frac{(q^{-1};q^3)_{2s}(q^{-1+2n};q^3)_{s}}
{(q^{3+2n};q^3)_{2s}(q^{3-2n};q^3)_{s}}q^{3s}
 \notag\\&\qquad\times
\sum_{k=0}^{(2n+1)/3}\frac{1-q^{6s-1+6k}}{1-q^{6s-1}}
\frac{(q^{6s-1},q^{-1-2n},q^{-1+2n+3s};q^3)_{k}}
{(q^{3},q^{3+2n+6s},q^{3-2n+3s};q^3)_{k}}q^{3k}.
\end{align}
Employing the replacements $a\mapsto q^{6s-1}$, $b\mapsto
q^{-1-2n}$, $c\mapsto q^{-1+2n+3s}$, $d\mapsto q^{3s+1}$, $q\mapsto
q^{3}$ in the identity \eqref{Dixon}, we get
\begin{align}
&\sum_{k=0}^{(2n+1)/3}\frac{1-q^{6s-1+6k}}{1-q^{6s-1}}
\frac{(q^{6s-1},q^{-1-2n},q^{-1+2n+3s};q^3)_{k}}
{(q^{3},q^{3+2n+6s},q^{3-2n+3s};q^3)_{k}}q^{3k}
 \notag \\&\quad\:=
\frac{(q^{2-2n},q^{6s+2};q^3)_{(2n+1)/3}}{(q^{3-2n+3s},q^{3s+1};q^3)_{(2n+1)/3}}
 \notag\\&\quad\:=0.
 \label{eq:wei-ggg}
\end{align}
Substituting \eqref{eq:wei-ggg} into \eqref{eq:wei-fff}, there is
the following equation:
\begin{align*}
&\sum_{k=s}^{(2n+1)/3+s}[6k-1]\frac{(q^{-1-2n};q^3)_{k-s}(q^{-1};q^3)_{k+s}(q^{-1+2n};q^3)_{k}}
{(q^{3};q^3)_{k-s}(q^{3+2n};q^3)_{k+s}(q^{3-2n};q^3)_{k}}q^{3k} =0.
\end{align*}
Hence we grasp hold of the $q$-congruence \eqref{eq:wei-eee} to
complete the proof.
\end{proof}

Now we are prepared to prove Theorem \ref{thm-b}.

\begin{proof}[Proof of Theorem \ref{thm-b}]
Since the polynomials $\Phi_n(q)$, $a-q^{2n}$, and $1-aq^{2n}$ in
$q$ are pairwise relatively prime, we can derive, from
\eqref{eq:wei-aaa}, \eqref{eq:wei-bbb}, and \eqref{eq:wei-eee}, the
following  $q$-supercongruence: for $s\leq (n-4)/3$, modulo
$\Phi_n(q)(a-q^{2n})(1-aq^{2n})$,
\begin{align*}
&\sum_{k=s}^{(2n+1)/3+s}[6k-1]\frac{(aq^{-1};q^3)_{k-s}(q^{-1};q^3)_{k+s}(q^{-1}/a;q^3)_{k}}
{(q^{3};q^3)_{k-s}(q^{3}/a;q^3)_{k+s}(aq^{3};q^3)_{k}}q^{3k}
\equiv0.
\end{align*}
It is clear that the factor
$(q^{3}/a;q^3)_{(2n+1)/3+2s}(aq^{3};q^3)_{(2n+1)/3+s}$ is prime to
$\Phi_n(q)$ when $a=1$ and $s\leq(n-4)$/6. Taking $a=1$ in this
$q$-supercongruence, we can deduce \eqref{eq:wei-b} in the end.
\end{proof}

\section{Declarations}
%%%%%%%%%%%%%%%%%%%%%%%%%%%%%%%%%%%%%%%%%%%%%%%%%%%%%%%%%%%%%%%%%%%%%%%%%%%%%%%%%%%%%%%%%%%%%%%%%%%%%%%%%%%%%%%%%
{\bf{Conflicts of interest:}} No potential conflict of interest was
reported by the authors.\\
{\bf{ Availability of data and material:}} Not
applicable. \\
{\bf{Code availability:}} Not applicable.\\
{\bf{Funding:}} The work is supported by the National Natural
Science Foundation of China (No. 12571349).

%%%%%%%%%%%%%%%%%%%%%%%%%%%%%%%%%%%%%%%%%%%%%%%%%%%%%%%%%%%%%%%%%%%%%%%%%%%%%%%%%%%%%%%%%%%%%%%%%%%%%%%%%%%%%%%%%%%%%%%%%%%%%%%%%%%%%%%%%%%%%%%%%%%%%%%%%%%%%%%%%%%%%%%%%%%%%%%%%%%%%%%%%%%%%%%%%%%%%%%%%%

\end{document}